\documentclass[12pt]{article}
\usepackage{latexsym,amsfonts,amssymb,amsmath,amsthm}
\usepackage{graphicx}

\usepackage[usenames,dvipsnames]{color}
\usepackage{ulem}

\begin{document}
\setlength{\baselineskip}{16pt}

\parindent 0.5cm
\evensidemargin 0cm \oddsidemargin 0cm \topmargin 0cm \textheight
22cm \textwidth 16cm \footskip 2cm \headsep 0cm

\newtheorem{theorem}{Theorem}[section]
\newtheorem{lemma}{Lemma}[section]
\newtheorem{proposition}{Proposition}[section]
\newtheorem{definition}{Definition}[section]
\newtheorem{example}{Example}[section]
\newtheorem{corollary}{Corollary}[section]

\newtheorem{remark}{Remark}[section]

\numberwithin{equation}{section}

\def\p{\partial}
\def\I{\textit}
\def\R{\mathbb R}
\def\C{\mathbb C}
\def\u{\underline}
\def\l{\lambda}
\def\a{\alpha}
\def\O{\Omega}
\def\e{\epsilon}
\def\ls{\lambda^*}
\def\D{\displaystyle}
\def\wyx{ \frac{w(y,t)}{w(x,t)}}
\def\imp{\Rightarrow}
\def\tE{\tilde E}
\def\tX{\tilde X}
\def\tH{\tilde H}
\def\tu{\tilde u}
\def\d{\mathcal D}
\def\aa{\mathcal A}
\def\DH{\mathcal D(\tH)}
\def\bE{\bar E}
\def\bH{\bar H}
\def\M{\mathcal M}
\renewcommand{\labelenumi}{(\arabic{enumi})}

\def\disp{\displaystyle}
\def\undertex#1{$\underline{\hbox{#1}}$}
\def\card{\mathop{\hbox{card}}}
\def\sgn{\mathop{\hbox{sgn}}}
\def\exp{\mathop{\hbox{exp}}}
\def\OFP{(\Omega,{\cal F},\PP)}
\newcommand\JM{Mierczy\'nski}
\newcommand\RR{\ensuremath{\mathbb{R}}}
\newcommand\CC{\ensuremath{\mathbb{C}}}
\newcommand\QQ{\ensuremath{\mathbb{Q}}}
\newcommand\ZZ{\ensuremath{\mathbb{Z}}}
\newcommand\NN{\ensuremath{\mathbb{N}}}
\newcommand\PP{\ensuremath{\mathbb{P}}}
\newcommand\abs[1]{\ensuremath{\lvert#1\rvert}}

\newcommand\normf[1]{\ensuremath{\lVert#1\rVert_{f}}}
\newcommand\normfRb[1]{\ensuremath{\lVert#1\rVert_{f,R_b}}}
\newcommand\normfRbone[1]{\ensuremath{\lVert#1\rVert_{f, R_{b_1}}}}
\newcommand\normfRbtwo[1]{\ensuremath{\lVert#1\rVert_{f,R_{b_2}}}}
\newcommand\normtwo[1]{\ensuremath{\lVert#1\rVert_{2}}}
\newcommand\norminfty[1]{\ensuremath{\lVert#1\rVert_{\infty}}}

\title{Semi-Wave Solutions of KPP-Fisher Equations with Free Boundaries in Spatially Almost Periodic Media. }

\author{Xing Liang \thanks{This work was supported by the the National Natural Science Foundation of China (11571334) and the Fundamental Research Funds for the Central Universities. }\\ School of Mathematical Sciences\\ University of Science and Technology of China
}
\date{}
\maketitle

\abstract
{In this paper we prove the existence and uniqueness 
%global stability
 of the almost periodic semi-wave (traveling wave) solutions of KPP-Fisher equations with free boundaries in spatially almost periodic media.}

\noindent \textbf{Keywords.}
Almost periodic media; KPP equation; Free boundary; Traveling waves\smallskip

\noindent \textbf{2010 Mathematics Subject Classification. 35K57; 	35B15; 35R35. }

\date{}
\maketitle
%\tableofcontents

\section{Introduction}

After the pioneer works of  \cite{Fisher, KPP} about traveling waves, and that  of \cite{AroWei2} about spreading speeds,  the propagation problems of reaction-diffusion equations in unbounded domain become one important branch of parabolic equations.
Specially,  traveling waves and spreading speeds of KPP-Fisher type and more general reaction-diffusion equations in heterogeneous media obtained more and more attentions of mathematicians in the last few decades.

As one simplest heterogenous case, the propagation problems in (spatially) periodic media were considered by mathematicians widely. Applying the approach of probability, \cite{GF1979} first proved the existence of  spreading speeds for one-dimensional KPP-Fisher equations in periodic media. 
\cite{SKT1986, xin1} gave the definition of the spatially periodic traveling waves independently, and then \cite{HZ95} proved the existence of the spatially periodic traveling waves of KPP-Fisher equations in the distributional sense.  
In the series of works (e.g.\cite{BH,BHR,BHN}), Berestycki, Hamel and their colleagues investigated the traveling waves and spreading speeds of KPP-Fisher type equations in high-dimensional periodic media deeply. Besides above works, a more general framework is given by \cite{LiZh,Wein2002} to study these two concepts for KPP-Fisher type equations and more general diffusion systems.

However, there are only a few works on the traveling waves and spreading speeds of KPP-Fisher equations in more complicated media. Applying  the theory of generalized principal eigenvalues, \cite{BN} proved the existence of the spreading speeds of KPP-Fisher equations in almost periodic or random stationary ergodic media. It is important to note that Matano \cite{matano} gave the definition of  spatially almost periodic traveling waves and provided some sufficient conditions on the existence of the spatially almost periodic traveling waves solution of  reaction-diffusion equations with bistable nonlinearity. We also point out that the propagation problems of  (temporally) nonautonomous reaction-diffusion equations were studied by Shen in many works(e.g. \cite{She1, She2, She3}). However, as we know, until now, there is no work on the traveling wave solutions of KPP-Fisher equations in almost periodic media.

Very recently, Du and Lin \cite{DuLi} first studied the propagation problems of reaction-diffusion equations with free boundaries. Though  a lot of works in this field appeared after \cite{DuLi}, most of them only considered the homogeneous media. The traveling waves and spreading speeds of reaction-diffusion equations with free boundaries in periodic media were only considered in \cite{DuLia, zh}. Once again,  no work considered the propagation problems in more complicated media.

The main aim of this paper is to prove the existence and uniqueness 
%and global stability
of the semi-wave solutions of the following diffusive KPP equation with a free boundary,
\begin{equation}
\label{main-eq}
\begin{cases}
u_t=u_{xx}+u(a(x)-u),\quad -\infty<x<h(t)\cr u(h(t),t)=0,h^{'}(t)=-\mu u_x(h(t),t)
\end{cases}
\end{equation}
where $\mu>0$, 
 $a(x)$ is a positive alomost periodic function in $x\in\RR$.
Here, as in previous works, we use the concept ``semi-wave" to replace "traveling wave" since the profile function of the wave is only defined on the half real line $(-\infty,0]$.

We introduce some concepts and basic results which will be used in this paper. Let $H(a)$ be the closure of $\{a(\cdot+t)|t\in \RR\}$ with the uniform topology.  
Given one metric space $\mathbb B$, we say one subset $\{\phi_g|g\in H(a)\}\subset \mathbb B$ is a {\bf one-cover} of $H(a)$ in $\mathbb B$ provided the mapping $F:g\to v_g$ is continuous.

We need consider not only the original equation \eqref{main-eq}, but also
 the following equation for any $g\in H(a)$.
\begin{equation}
\label{main-eq-g}
\begin{cases}
u_t=u_{xx}+u(g(x)-u),\quad -\infty<x<h(t)\cr u(h(t),t)=0,h^{'}(t)=-\mu u_x(h(t),t)
\end{cases}
\end{equation} 

As usual, we first consider the steady states and dynamics of the respective equation on $\RR$ \begin{equation}
\label{equationR}
u_t=u_{xx}+u(g(x)-u),\quad x\in \RR, g\in H(a),
\end{equation} 

 \begin{proposition}
 \label{positivesteadystate}For any $g\in H(a)$, \eqref{equationR} has a unique bounded positive steady state $u_g^*$,  and $u^*_{ g(\cdot+s) }(x)=u^*_g(x+s)$ for any $s\in \RR$. Moreover,  define one mapping $F:H(a)\to L^\infty(\RR)\cap C(\RR)$ by $F(g) = u^*_g$. Then $F$ is continuous, that is, $\{u^*_g|g\in H(a)\}$ is a one-cover of $H(a)$ in $L^\infty(\RR)\cap C(\RR)$.
Specially, $u_g^*$ is one almost periodic function.
 
 For any continuous initial data $u_0$ , if there are two positive numbers $m<M$ such that 
 $0<m<u_0(x)<M$ for any $x\in \RR$,  then the solution of the Cauchy problem of \eqref{equationR} converges to $ u^*_g$ as $t\to +\infty$ uniformly for $x\in \RR$ .  \end{proposition}
Now we give the definition of the almost periodic semi-wave solution of \eqref{main-eq-g}:
\begin{definition}\label{traveling-wave}
Fix $g\in H(a)$. Let $(u(x,t;g),h(t;g))$ be one positive bounded entire solution of \eqref{main-eq-g}.  We call $(u(x,t;g),h(t;g))$ one almost periodic  semi-wave solution of \eqref{main-eq-g}, simply, semi-wave solution, provided $u(x,t;g)$ satisfies the following hypotheses: \begin{enumerate}\item 
There is some  $v=v(\xi,\tau;g)\in C^2((-\infty,0]\times \RR)$ such that $u(x,t;g)$ can be written as $u(x,t;g)=v(x-h(t;g),h(t;g);g)$, where $h(\pm\infty;g)=\pm \infty$. 
\item Denote $v_{g(\cdot+s)}(\xi):=v(\xi,s;g)$. Then $v_{g(\cdot+s_n)}(\cdot)$ converges in $L^\infty_{local}((-\infty,0])\cap C((-\infty,0])$  provided $g(\cdot+s_n)\to \bar g$ in $H(a)$.  
\item Using $v_{\bar g}$ to denote the limit of $v_{g(\cdot+s_n)}(\xi)$, $\{v_g|g\in H(a)\}$ is a one-cover of $H(a)$. \end{enumerate} \end{definition}
Specially, from this definition, $v=v(\xi,\tau;g) $ is one almost periodic function in $\tau$ from $\RR$ to $L^\infty_{local}((-\infty,0])\cap C((-\infty,0])$.

We can also give the following  equivalent definition of the almost periodic semi-wave to Definition \ref{traveling-wave}.
\begin{definition}\label{traveling-wave2}
Let $\{(u(x,t;g),h(t;g))|g\in H(a)\}$ be one family of functions. Suppose for each $g\in H(a)$ , $(u(x,t;g),h(t;g))$ is one positive bounded entire solution of \eqref{main-eq-g}.  We call $(u(x,t;g),h(t;g))$ one almost periodic  semi-wave solution of \eqref{main-eq-g} provided $u(x,t;g)$ satisfies the following hypotheses: \begin{enumerate}\item There is some  $v=v(\xi,\tau;g)\in C^2((-\infty,0]\times \RR)$ such that $u(x,t;g)$ can be written as $u(x,t;g)=v(x-h(t;g),h(t;g);g)$, where $h(\pm\infty;g)=\pm \infty$. 
\item
For any $s\in \RR$, $ v (\xi, \tau+s;g)= v (\xi, \tau;g(\cdot+s)), \xi\leq 0,\tau\in \RR $. 
\item $ \{v (\cdot,\cdot, ;g)\}$ is a one-cover of $H(a)$ in $L^\infty_{local}((-\infty,0]\times \RR)\cap C((-\infty,0]\times \RR)$. \end{enumerate} \end{definition}

\cite{matano} first gave the definition of (spatially) almost periodic solution. Roughly,  in \cite{matano} one entire solution $u$ of some reaction-diffusion equation is said to be one almost periodic traveling wave provided $u(x,t):=w(x-h(t),x)$ and $w$ is almost periodic in $x$. Such a definition is equivalent to ours since if $(u(x,t;g),h(t;g))$ is one semi-wave solution in our definition, then we also can rewrite $u(x,t):=w(x-h(t),x)$ for some function $w(\eta,x)$ and then the profile function $v$ in our definition satisfies  $v(\xi,\tau;g) =w(\xi, \xi+\tau)$. The almost periodic property of $v$ in $\tau$ is equivalent to that of $w$ in $x$.

Now, we show the main theorem of this paper.

\begin{theorem}\label{maintheorem}We have the following conclusions on the existence and uniqueness %and global stability 
of the almost periodic semi-wave solution of \eqref{main-eq-g}:

\begin{enumerate}
\item \eqref{main-eq-g} has one semi-wave solution $(\tilde u,\tilde h)$. Furthermore,  $(\tilde u,\tilde h)$ has the following three properties: \begin{enumerate}
\item[(a)] Let $\tilde u(x,t)=v(x-\tilde h(t),\tilde h(t))$. Then $ v=v(\xi, \tau)$ is an almost periodic function in $\tau$ from $\RR$ to $L^\infty((-\infty,0])\cap C((-\infty,0])$. 
\item[(b)] $|\tilde u(x,t)- u^*_g(x)|\to 0$ as $x\to -\infty$ locally uniformly for $t\in\RR$. In this sense, we say that such a semi-wave solution connects $u^*$ and $0$. More precisely, 
 $|v(\xi, \tau)- u^*_g(\xi+\tau)|\to 0$ as $\xi\to -\infty$ uniformly for $\tau\in\RR$. 
 \item[(c)] $\tilde u(x,t)\to u^*_g(x)$ as $t\to +\infty$ uniformly for $x\in (-\infty,M]$ where $M$ can be any 
constant in $\RR$.
\end{enumerate}

\item The almost periodic semi-wave solution of \eqref{main-eq-g} connecting $u^*_g$ and $0$ is unique up to the time translation.
\item  Let $(u(x,t),h(t))$ be one positive entire solution of  \eqref{main-eq-g} with $u(x,t)\leq u_g^*(x)$ for any $x\in (-\infty,h(t)], t\in \RR$. If there are some $\epsilon,\delta>0$ such that $u(x,t)>\epsilon$ for any $x\in (-\infty,h(t)-\delta], t\in \RR$,  then $u(x,t)\equiv\tilde u(x,t+t_0)$ for some $t_0\in \RR$.
%\item The semi-wave solution $(\tilde u,\tilde h)$ is globally stable in the following sense:\\
%Consider the solution $(u,h)$ of \eqref{main-eq} with the positive bounded initial data $(u_0,h_0)$. Then $|u(x+h(t),t)- \tilde u(x+\tilde h(t), t')|\to 0 $ as $t,t'\to\infty $ locally uniformly for $x\in (-\infty,0]$ provided $h(t)=\tilde h(t')$. 
\end{enumerate}

\end{theorem}

\section{Proof of the main results}
\begin{proof}[Proof of Proposition \ref{positivesteadystate}]  The existence and uniqueness of the positive bounded steady state of \eqref{equationR} is well-known, e.g. see 
\cite{DLiu}. The conclusion that $u^*_{ g(\cdot+s) }(x)=u^*_g(x+s)$ is also easy to be checked. We consider the almost periodic property of $u^*_g$. We know that there are some $s,S\in \RR$ with $0<s<S$ such that $s\leq g(x)\leq S, x\in \RR,g\in H(a)$. Hence, there are some $s',S'\in \RR$ with $0<s'<s<S<S'$ such that $s$ is a lower solution and $S'$ is an upper solution. Therefore, $s'\leq u^*_g(x)\leq S', x\in \RR,g\in H(a)$. Moreover, $(u^*_g)_x,(u^*_g)_{xx}$ are uniform bounded for $x\in \RR,g\in H(a)$. Let $\bar g=\lim\limits_{n\to \infty} g_n$ for some sequence $\{g_n\}$.  Then by choosing a subsequence of $\{u^*_{g_n}\}$, still denoted by $\{u^*_{g_n}\}$,  $\lim\limits_{n\to\infty} u^*_{g_n}$ exists in $L^\infty_{local}(\RR)\cap C(\RR)$.  Moreover, $v=\lim\limits_{n\to\infty} u^*_{g_n}$is the unique positive bounded steady state of $u_t=u_{xx}+u(\bar g(x)-u), x\in \RR$. This means that $v=u_{\bar g}^*$. Now, we find that $\bar H(u^*)=\{u^*_g|g\in H(a)\}$ is a one-cover of $H(a)$ in $L^\infty_{local}(\RR)\cap C(\RR)$.  We show that $\bar H(u^*)$ is also a one-cover of $H(a)$ in $L^\infty(\RR)\cap C(\RR)$. If not, there are some $g_n,\bar g\in H(a), n=1,2,\cdots$ with $\bar g=\lim\limits_{n\to \infty} g_n$ ,  and some $x_n\in \RR, n=1,2,\cdots$ such that $|u^*_{g_n}(x_n)-u^*_{\bar g}(x_n)|\geq \epsilon$ for some fixed $\epsilon>0$ and that $g_n(\cdot+x_n), \bar g(\cdot+x_n)\to g_0$ as $n\to \infty$. This means $|u^*_{g_n(\cdot+x_n)}(0)-u^*_{\bar g(\cdot+x_n)}(0)|\geq \epsilon$. On the other hand,  $u^*_{g_n(\cdot+x_n)}(0), u^*_{\bar g(\cdot+x_n)}\to u^*_{g_0}(0)$ as $n\to \infty$. It is a contradiction. 
Therefore, $\bar H(u^*)$ is also a one-cover of $H(a)$ in $L^\infty(\RR)\cap C(\RR)$. $u^*_g$ is almost periodic.

Let us consider the stability now. It is easy to see that $su^*$ is a lower solution of  \eqref{equationR} for any $0<s<1$ and an upper solution of \eqref{equationR} for any $s>1$. Hence, the solution $u(x,t)$ of  \eqref{equationR} with the initial data $su^*$ converges to one positive steady state as $t\to \infty$ which is just $u^*$. By the comparison principle, $u^*$ is globally stable. 
 \end{proof}

In this paper, when we consider the initial  data $(u_0, h_0)$ of  equation \eqref{main-eq-g}, we always assume that $h_0\in \RR$ and $u_0\in C((-\infty,h_0])$ is bounded. We say that the initial data $(u_0,h_0)$ is positive provided $u_0>0$ for $x<h_0$.
In \cite{ddl}, it is proved
\begin{proposition}[Theorem 2 .11, \cite{ddl} ]\label{estimate}
For any $g\in H(a)$ and any positive initial data $(u_0,h_0)$, the solution $( u, h)$ of \eqref{main-eq-g} with  initial data $(u_0,h_0)$ exists for $t>0$, and $h\in C^{1}\big((0,+\infty)\big)\cap C\big([0,\infty)\big)$, $u\in C^{1,2}(G_+)\cap C\big(\overline{G_+}\big)$ with $G_+=\big\{(x,t)\in\R^2:\,t\in (0,\infty),\,x\in (-\infty,h(t)] \big\}$.
Furthermore, for any $T>\tau>0$ and any $A\leq h_0$, there holds
\begin{equation*}
\big\|u\big\|_{C^{(1+\alpha)/2,1+\alpha}(G_{A,T}^{\tau})}+\big\|h\big\|_{C^{1+\alpha/2}([\tau,T])} \leq C,
\end{equation*}
where $G_{A,T}^\tau=\big\{(x,t)\in\R^2:\,t\in[\tau,T],\,x\in[A,h(t)] \big\}$, and $C$ is a positive constant depending on $\tau$, $\|u_0\|_{L^{\infty}((-\infty,h_0])}$. 
\end{proposition}

%\begin{lemma}\label{lemma100}
%Let $( u, h)$ be the solution of \eqref{main-eq-g} with the positive initial data $(u_0,h_0)$, then $( u(x+h_0,t),  h(t)-h_0)$ is the solution of  \begin{equation}
%\begin{cases}
%u_t=u_{xx}+u(g(x+h_0)-u),\quad -\infty<x<\bar h(t)\cr u(t,\bar h(t))=0,h^{'}(t)=-\mu u_x(t,h(t))
%\end{cases}
%\end{equation} 
%with the initial data $(u_0(\cdot+h_0),0)$.
%\end{lemma}
%We always let $g\in H(a)$ and use $(u(x,t;u_0, h_0,g), h(t;u_0,h_0,g)) $ to denote the solution of  \eqref{main-eq-g} with the initial data $(u_0,h_0)$.  
%Let $\underline a=\inf _{x\in \RR}a(x)$.  
Let $\underline a=\inf_{x\in \RR}a(x)$. It is known that $0<\underline a=\inf_{x\in \RR}g(x),\forall g\in H(a)$ and that for any positive initial data $(u_0,h(0))$, the solution $(u,h)$ of\begin{equation}
\label{underlinea}\begin{cases}
u_t=u_{xx}+u(\underline a-u),\quad -\infty<x<h(t)\cr u(t,h(t))=0,h^{'}(t)=-\mu u_x(t,h(t))
\end{cases}
\end{equation}
satisfies that $h(t)/t\to c$ as $t\to +\infty$ for some $c>0$(\cite{DuLi}).
Then by the comparison principle  (e.g. see \cite{DuLi}), we have
\begin{lemma}\label{lemma}Let $(u,h)$ be the solution of \eqref{main-eq-g} with the positive initial data $(u_0,h(0))$. Then $\limsup_{t\to \infty}h'(t)\geq c$.\end{lemma}
%This lemma comes from the comparison principle and the fact $0<m:=\inf a$ and the classical result about the spreading-vanishing dichotomy of the same problem but %with constant coefficients .

To emphasize the coefficient $g\in H(a)$, we use $(u(x,t;u_0, h_0,g), h(t;u_0,h_0,g)) $ to denote the solution of  \eqref{main-eq-g} with the initial data $(u_0,h_0)$.
\begin{lemma}\label{lemma1}
 Let $g\in H(a)$ and $(u(x,t;u_0, h_0,g), h(t;u_0,h_0,g)) $ be the solution of  \eqref{main-eq-g} with the positive initial data $(u_0,h_0)$. Then $\liminf_{t\to +\infty}h'(t;u_0,h_0,g)>0$.\end{lemma}
\begin{proof}
Suppose, for the sake of contradiction, $\liminf_{t\to +\infty}h'(t;u_0,h_0,g)=0$. That is, there is a sequence $\{t_n\}$ with $t_n\to +\infty$ as $n\to \infty$ such that $h'(t_n;u_0,h_0,g)\to 0$.  Lemma \ref{lemma} implies that there is a positive number $\epsilon$  such that $\limsup_{t\to \infty}h'(t;u_0,h_0,g)>\epsilon$. Then we choose another sequence $\{s_n\}$ such that $h'(s_n;u_0,h_0,g)=\epsilon$, $s_n<t_n<s_{n+1}, n\in \mathbb N$ and $h'(\tau;u_0,h_0,g)\leq \epsilon, \tau\in [s_n,t_n]$. Let $v^n(x,t)=u(x+h(s_n;u_0,h_0,g), t+s_n;u_0,h_0,g)$, $h^n(t)=h(\cdot+s_n;u_0,h_0,g)-h(s_n;u_0,h_0,g)$. $(v^n, h^n)$ is the solution of
\begin{equation}
\begin{cases}
v^n_t=v^n_{xx}+v^n(g(x+h(s_n;u_0,h_0,g))-v^n),\quad -\infty<x<h^n(t)\cr v^n(h^n(t),t)=0,(h^n)'(t)=-\mu v^n_x(h^n(t),t)
\end{cases}
\end{equation} with the initial data $(v^n(\cdot,0), 0)$ and this solution can be extended to $t\in (-s_n,+\infty)$.
By the priori estimates, we can find one subsequence of $(v^n,h^n)$, still denoted by $(v^n,h^n)$  such that there is some $(v,\bar h)$ satisfying $\bar h(0)=h^n(0)=0$, $h^n\to\bar h$ in $C^1([x_1,x_2])$ for any bounded interval $[x_1,x_2]$, $v^n\to v$ in $C^2 (\Omega)$  with $\Omega$ being any compact domain of $G_+=\{(t,x)|x\in (-\infty, \bar h(t)], t\in \RR\}$.  Moreover, we can suppose $g(x+h(s_n;u_0,h_0,g))\to \bar g(x)$ uniformly on $\RR$ and then $(v,\bar h)$ solves \begin{equation}
\begin{cases}
v_t=v_{xx}+v(\bar g(x)-v),\quad -\infty<x< \bar h(t)\cr v(\bar h(t),t)=0, \bar h{'}(t)=-\mu v_x(\bar h(t),t)
\end{cases}
\end{equation} with the initial data $(v(\cdot,0), 0)$.  Notice that $\bar h'(0)=-v_x(0,0)=\epsilon >0$ and $(v,\bar h)$ is an entire solution. The strong comparison principle implies $v(x,0)>0$ for $x\in (-\infty,\bar h(0))$.  If  $\{T_n=t_n-s_n\}$ has a bounded converging subsequence, still denoted by $\{T_n\}$, $T_n\to T$ as $n\to \infty$, then the Hopf lemma shows that $\bar h'(T)>0$. On the other hand, that $h'(t_n)\to 0$ as $n\to \infty$ implies that $\bar h'(T)=0$, it is a contradiction.  Hence, $T_n=t_n-s_n\to +\infty$ and then $\bar h'(t)\leq \epsilon$ for $t\in [0,+\infty)$.  However, let $(\underline v, \underline h)$ be the solution of
\eqref{underlinea}
with the initial data $(v(\cdot,0), 0)$. Then the comparison principle of free boundary problems yields that $\underline h(t)\leq \bar h(t)\leq \epsilon t$.  On the other hand, the work in \cite{DuLi} shows that $\underline h(t)/t\to c>0$ as $t\to +\infty$ for some positive number $c$. Choose $\epsilon<c$, we obtain one contradiction. The proof is complete. 
\end{proof}
Let 
$$
\mathcal H(x)=\begin{cases}1,\quad &x\leq -1\cr
-x,\quad &-1<x\leq 0
\end{cases}
$$
For one given $g\in H(a)$, define $u_{h_0,g}(x)$ and $u_{h_0,n,g}(x)$ by
$$
u_{h_0,g}(x)=\begin{cases} u^*_g(x),\quad &x<h_0\cr
0,\quad &x=h_0
\end{cases}
$$
and
$$
u_{h_0,n,g}(x)=\mathcal H(n(x-h_0))u_{h_0,0,g}(x).
$$
Then
$$
u_{h_0,n,g}(x)\ge u_{h_0,m,g}(x),\quad  x\le h_0,\ \forall n>m
$$
and
$$
u_{h_0,n,g}(x)\to u_{h_0,g}(x),\quad \forall\,\, x\le 0
$$
as $n\to\infty$. 

Let $(u(t,x; g,n,h_0), h_{g,n,h_0}(t))$ be the solution of  \eqref{main-eq-g} with the initial data $(u_{h_0,n,g}(x), h_0)$ for $g\in H(a)$.
\begin{lemma}\label{lemma2.4}For any $\epsilon>0$,  there are  two positive numbers $m<M$ such that $m<h'_{g,n,h_0}(t)<M $ for any $g\in H(a),$ positive integer $n$, $h_0\in \RR$ and $t\in [\epsilon,+\infty)$.
\end{lemma}

\begin{proof}
The existence of the upper bound comes from Proposition \ref{estimate}. The proof of the existence of the lower bound is the same as Lemma \ref{lemma1}.
%By Lemma \ref{lemma1}, we can only consider the case that $h_0=0$. To be simple, let $(u(t,x; g,n), h_{g,n}(t))= (u(t,x; g,n,0), h_{g,n,0}(t))$. 
%Suppose  $w(t,x; g,n)=u(t,x+h_{g,n}(t); g,n)$ for $t\geq0,x\leq 0$.  Then $w_t=w_{xx}+h'_{g,n}(t)w_x+w(g(x+h_{g,n}(t))-w)$.  Let $\bar w $ be the weak solution of  $w_t=w_{xx}+w(\bar a-w)$  for $t>0, x<0$ with $w(t,0)=0$ and the initial data $%%\bar w(0,x)=C>u^*(x)$  for any $x\in \RR $. Then the comparison principle yields $w(t,x; g,n)\leq \bar w(t,x)$ for $t>0, x\leq 0$.  Hence  we can choose $M=\max _{t\geq \epsilon}\{-\mu\bar w_x(t,0)\}$. 

%Moreover,  by the Holder estimation, we have $u_t(t,x;g,n), u_x(t,x;g,n), u_{xx}(t,x;g,n)$ are bounded for $t>\epsilon$.  Then repeating the proof of  Lemma \ref{lemma1}, we can prove this lemma.

\end{proof}
The comparison principle implies that $u(t,x; g,n,h_0)$ and $h_{g,n,h_0}(t)$ are both increasing in $n$.  Let $\bar u (t,x;g,h_0)=\lim_{n\to \infty} u(t,x; g,n,h_0)$ and $\bar h(t;g,h_0) =\lim_{n\to \infty} h_{g,n,h_0}(t)$, we have 
 \begin{lemma}
$( \bar u (t,x;g,h_0), \bar h(t;g,h_0))$ solves \eqref{main-eq-g} for $t>0$. 
 \end{lemma}
 Now we use the method of pulling back to transfer the equation with a free boundary to the one with a fixed boundary.
\begin{lemma} \label{lemma2}If $(u(x,t), h(t)) $ is one solution of  \eqref{main-eq-g} 
for $t\in (t_0,+\infty), t_0\in [-\infty,+\infty)$, then $v(\xi, \tau ):=u(\xi+\tau,h^{-1}(\tau))$ is one solution of 
 \begin{equation}
\label{veq}
\begin{cases}
-v_\xi(0,\tau)(-v_\xi+v_\tau)=v_{\xi\xi}+v(g(\xi+\tau)-v),\quad -\infty<\xi<0,\tau>h_0\cr v(0,\tau)=0,\tau>h_0,
\end{cases}
\end{equation} 
for $\tau>h_0$. Here $h_0=\lim_{t\to t_0}h(t)$.  Moreover, if $t_0>-\infty$, and $u(x,t)\to u_0(x)$ as $t\to t_0$ locally uniformly, then $v(\xi,\tau)\to u(\xi+h_0)$ as $\tau\to h_0$ locally uniformly.

On the other hand,  suppose that $h_0\in [-\infty,+\infty)$, $v(\xi,\tau)$ is one solution of \eqref{veq} for $\tau>h_0$ and $v_\xi(0,\tau)<0, \tau>h_0$. Let $h(t)$ be the solution of $h'(t)=-\mu v_\xi(0,,h(t))$ with $h(t_0)=h_0$, $u(x,t)=v(x-h(t),h(t))$. Then $(u(x,t),h(t))$ is one solution of  \eqref{main-eq-g}. Specially, $h_0=-\infty$ means $t_0=-\infty$.

Finally, that $(u,t)$ is one almost periodic semi-wave  solution of \eqref{main-eq-g} is equivalent to that $v=v(\xi,\tau)$ is almost periodic in $\tau$ from $\RR$ to $L^\infty_{local}((-\infty,0])\cap C((-\infty,0])$.
\end{lemma}
%It is easy to see that $v (\xi,\tau+s;g)=v(\xi,\tau;g(\cdot+s))$.

By Lemma \ref{lemma2},  $( \bar u (t,x;g,h_0), \bar h(t;g,h_0))$ can generate one solution $v(\xi,\tau;g,h_0)$ of \eqref{veq} for $\tau>h_0$. Let $h_0\to -\infty$. By choosing one subsequence, we have $\tilde v(\xi,\tau;g)=\lim_{n\to \infty}v(\xi,\tau;g,h_0^n)$ as $h_0^n\to -\infty$ locally uniformly for $\xi\leq 0, \tau\in \RR$. We will show later that $\tilde v$ does not depend on the choice of the subsequence. 
\begin{lemma}
$\tilde v(\xi,\tau;g)$ is one entire solution of  \eqref{veq}. 
\end{lemma}
Furthermore, let $ \tilde h(t;g)$ be one solution of the ordinary differential equation $\tilde h'(t;g)=-\mu\frac{\partial \tilde v(\xi,\tilde h(t;g);g)}{\partial \xi}|_{\xi=0}$ and $\tilde u(x,t;g)=\tilde v(x-\tilde h(t;g),\tilde h(t;g);g)$ for $t\in \RR$ and $x\leq \tilde h(t;g)$.
In what follows, we will show that  $(\tilde u(x,t;g), \tilde h(t;g))$  is the unique almost periodic semi-wave solution of  \eqref{main-eq-g} connecting $u^*_g$ and $0$.  First of all, 
\begin{lemma}\label{lemma8} $(\tilde u(x,t;g), \tilde h(t;g))$ is one entire solution of \eqref{main-eq-g}.  Moreover, $\bar u(t_n,x;g,h_0^n)\to \tilde u(x,t_0;g)$ locally uniformly for $x\leq \tilde h(t_0;g)$ provided $\bar h (t_n;g,h_0^n)\equiv \tilde  h(t_0;g)$ for any $t_0\in \RR$. 
%Moreover, for any $t_0\in \RR$, $ \tilde h(t+t_0;g)$ satisfy $\frac{ d\tilde h(t+t_0;g)}{dt}=-\mu\frac{\partial \tilde v(\xi,\tilde h(t+t_0;g);g)}{\partial \xi}|_{\xi=0}$ and $(\tilde u(x,t+t_0;g), \tilde h(t+t_0;g))$ is one entire solution of 
%%\eqref{main-eq-g}.  
\end{lemma}

Then, by the theory of  zero number of parabolic equations, we have

\begin{proposition}\label{prop1}Let $g\in H(a)$ and $(u(x,t),h(t))$  be a positive entire solution of \eqref{main-eq-g}. Suppose that  $u(x,t)< u^*_g(x)$ for any $t\in \RR, x\leq h(t)$. Then $u(x,t_2)\leq \tilde u(x,t_1;g)$ for any $x\leq h(t_2)$ provided $h(t_2)\leq \tilde h(t_1;g)$. Specially,  $ \tilde u(x,t;g)$ is increasing in $t$. 
\end{proposition}
\begin{proof}
Since $(u(x,t),h(t))$ is one bounded entire solution, $u_x(x,t)$ is uniform bounded (see Proposition \ref{estimate}). Hence there is some integer $N$ such that for any $n>N$, and any $t\in \mathbb R$ $$
\begin{cases}u_{h_0,n,g}(x)\geq u(x,t),\ for\ x\leq h(t), \; & if\  h(t)\leq h_0;\cr
\exists \ x_0<h_0 \begin{cases} u_{h_0,n,g}(x)>u(x,t), \ for\ x< x_0;\cr\  u_{h_0,n,g}(x)<u(x,t), \ for\ x _0< x\leq h_0\end{cases}\quad &if \ h(t)>h_0.
\end{cases}
$$

Hence, by the theory of the zero numbers of parabolic equations (e.g. see Theorem D in \cite{ASB} or Lemma 4.1 in \cite{lls}), for any $n>N$, $t_1>0$, $t_2\in \mathbb R$
$$
\begin{cases}u(t_1,x; g,n,h_0)\geq u(x,t_2),\ for\ x\leq h(t_2), \; & if\  h(t_2)\leq  h_{g,n,h_0}(t_1);\cr  \exists\ x_0<h_0 \begin{cases}
u(t_1,x; g,n,h_0)>u(x,t_2), \ for\ x< x_0;\cr  u(t_1,x; g,n,h_0)<u(x,t_2), \ for\ x_0< x\leq h_{g,n,h_0}(t_1) \end{cases}\quad &if h(t_2)\geq  h_{g,n,h_0}(t_1),
\end{cases}
$$
where $x_0$ depends on $n, t_1,t_2$.
Let $n\to \infty$,
\begin{equation}\label{above1}
\begin{cases}\bar u(t_1,x; g,h_0)\geq u(x,t_2),\ for\ x\leq h(t_2), \; & if\  h(t_2)\leq  \bar h(t_1;g,h_0);\cr  \exists\ x_0<h_0 \begin{cases}
\bar u(t_1,x; g,h_0)>u(x,t_2), \ for\ x< x_0;\cr  \bar u(t_1,x; g,h_0)<u(x,t_2), \ for\ x_0< x\leq  \bar h(t_1;g,h_0) \end{cases}\quad &if h(t_2)\geq   \bar h(t_1;g,h_0),
\end{cases}
\end{equation} where $x_0$ depends on $t_1,t_2$.

Notice that \eqref{above1} holds for any $h_0\in \mathbb R$. Take a subsequence of $h_0\to -\infty$,   $\tilde u(t_1,x; g)\geq u(x,t_2)$ for $\ x\leq h(t_2)$ if $\  h(t_2)\leq  \tilde h(t_1;g)$ by Lemma \ref{lemma8}.
\end{proof}

\begin{lemma}\label{lemma2.9}
 There are some $\epsilon,\delta>0$ such that $\tilde u(\tilde h(t;g)-\delta,t;g)\geq \epsilon$ for $t\in \RR$ and $g\in H(a)$.  Moreover, $|\tilde u(x+\tilde h(t),t;g)- u^*_g(x+h(t))|\to 0$ as $x+\tilde h(t)\to -\infty$ uniformly for $t\in \RR$ and $g\in H(a)$. And hence, $|\tilde v(\xi,\tau;g)-u^*_{g(\cdot+\tau)}(\xi)|\to 0$ as $\xi\to -\infty$ uniformly for $\tau\in \RR$ and $g\in H(a)$.
\end{lemma}
\begin{proof}

First, note that $\tilde h'(t;g)>0$ uniformly for $t\in \RR$ and $g\in H(a)$ by Lemma \ref{lemma2.4} and that $\tilde u_{xx}(x,t;g)$ is uniformly bounded for $x\in (-\infty,\tilde h(t;g)]$, $t\in \RR$ and $g\in H(a)$. Hence there are some $\epsilon,\delta>0$ such that $\tilde u(\tilde h(t;g)-\delta,t;g)\geq \epsilon$ for $t\in \RR$ and $g\in H(a)$. It is also known that $\tilde u(x,t;g)$ is increasing in $t$ by Proposition \ref{prop1}. Hence, $\tilde u(x,t;g)\geq \epsilon$ for $x\leq \tilde h(t;g)-\delta$,$t\in \RR$ and $g\in H(a)$. This means that $\tilde v (\xi,\tau;g)\geq \epsilon$ for $\xi<-\delta, \tau\in \RR$ and $g\in H(a)$.

Assume, for sake of contradiction, that there are some positive constant C and sequences $\{x_n\}_{n=0}^\infty, \{t_n\}_{n=0}^\infty, \{g_n\}_{n=0}^\infty$ with $x_n+h(t_n)\to -\infty$ as $n\to \infty$ such that $|\tilde u(x_n+\tilde h(t_n),t_n;g_n)-u^*_{g_n}(x_n+h(t_n))|> C$. Without loss of generality, we can suppose that there is some $\bar g\in H(a)$ such that $u^*_{g_n}(x+x_n+h(t_n))\to u^*_{\bar g}(x)$ uniformly for $x\in \RR$ and  $\tilde u(x+x_n+\tilde h(t_n),t+t_n;g_n)\to \psi(x,t)$ locally uniformly for $x,t\in \RR$.  $\psi$ is one bounded entire solution of $u_t=u_{xx}+u(\bar g(x)-u)$ and $\psi(x,t)>\epsilon$ for $x,t\in \RR$.  This means that $\psi(x,t)=u^*_{\bar g}(x)$. On the other hand $\psi(0,0)=\lim_{n\to \infty}\tilde u(x+x_n+\tilde h(t_n),t+t_n;g_n)$. This implies that $|\psi(0,0)-u^*_{\bar g}(0)|\geq C$, a contradiction. The proof is complete.
\end{proof}
Let us go back to the equation \eqref{veq}.  From Lemma \ref{lemma1} and Proposition \ref{prop1}, we have 
\begin{lemma}
Let $v(\xi,\tau)$ be one entire solution of \eqref{veq} generated by $u(x,t)=v(x-h(t), h(t))$ where $(u,h)$ is a positive entire solution of \eqref{main-eq-g} and $u(x,t)\leq u^*_g(x)$. Then $\tilde v(\xi,\tau;g)\geq v(\xi,\tau)$.
\end{lemma}
Define $\rho(\tau, v_1,v_2;g)=\inf_{\xi<0}\frac{v_1(\xi,\tau)}{v_2(\xi,\tau)}$, where $v_1,v_2$ are two entire solutions of \eqref{veq}. 
We have 
\begin{lemma} \label{lemma10}Let $v_1,v_2$ be two entire solutions of \eqref{veq}.  Suppose that $0<v_1(\xi,\tau)\leq v_2(\xi,\tau) \leq u^*_g(\xi+\tau)$ and $-(v_2)_\xi+(v_2)_\tau\geq 0$  for $\xi<0$, $\tau\in \mathbb R$ and that $(v_2)_\xi(0,\tau)<0$, $0<\rho(\tau, v_1, v_2;g)<1$ for $\tau\in \RR$.  Then $\rho(\tau, v_1, v_2;g)$ is increasing in $\tau$. In addition, if $\liminf_{\xi\to\infty}\frac{v_1(\xi,\tau)}{v_2(\xi,\tau)}>\rho(\tau, v_1, v_2;g)$ at some $\tau_0$, then $\rho(\tau, v_1, v_2;g)>\rho(\tau_0, v_1, v_2;g)$ for any $\tau>\tau_0$.
\end{lemma}
\begin{proof}
Without loss of generaility, we only prove that $\rho(0, v_1, v_2;g)\leq \rho(\tau, v_1,v_2;g)$ for any $\tau>0$. Let $\rho_0=\rho(0, v_1, v_2;g)$. Then $\rho_0 v_2(\xi,0;g)\leq v_1(\xi,0), \forall \xi\leq 0$. Let $w=\rho_0 v_2$. Then
  \begin{equation}
\label{weq}
\begin{cases}
- (v_2)_\xi(0,\tau)(-w_\xi+w_\tau)=w_{\xi\xi}+w(g(\xi+\tau)- v_2),\quad -\infty<\xi<0,\tau>0\cr u(0,\tau)=0,\tau>0.
\end{cases}
\end{equation} From the assumption, we know that 
  $-(v_2)_\xi(0,\tau)(-w_\xi+w_\tau)\geq 0$ and $-(v_2)_\xi(0,\tau)\geq -(v_1)_\xi (0,\tau)>0$. Hence,
   \begin{equation}
\label{w1eq}
-(v_1)_\xi(0,\tau)(-w_\xi+w_\tau)\leq w_{\xi\xi}+w(g(\xi+\tau)-v_1),\quad -\infty<\xi<0,\tau>0.
\end{equation}  
This means that $w$ is one lower solution of  the following equation of $\phi$ \begin{equation}
\label{phieq}
\begin{cases}
- (v_1)_\xi(0,\tau)(-\phi_\xi+\phi_\tau)=\phi_{\xi\xi}+\phi(g(\xi+\tau)- v_1),\quad -\infty<\xi<0,\tau>0\cr u(0,\tau)=0,\tau>0.
\end{cases}
\end{equation} 
We also know that $v_1$ is one solution of \eqref{phieq} and $w(\xi,0)\leq v_1(\xi,0),\forall \xi\leq 0 $.  Therefore, it holds that $w(\xi,\tau)\leq v_1(\xi,\tau),\forall \xi< 0,\tau>0 $. 
In addition, if $\liminf_{\xi\to\infty}\frac{v_1(\xi,0)}{v_2(\xi,0)}>\rho(0, v_1, v_2;g)$, then $\liminf_{\xi\to\infty}\frac{v_1(\xi,\tau)}{v_2(\xi,\tau)}>\rho(0, v_1, v_2;g)$ for small $\tau>0$ and $w(\xi,0)\leq v_1(\xi,0)$ at some$\xi\leq 0 $. Hence,  $\rho(\tau, v_1, v_2;g)>\rho(0, v_1, v_2;g)$ for any $\tau>0$. The proof is complete.
\end{proof}
Before to prove the Theorem \ref{maintheorem}, we show one simple Lemma:
\begin{lemma}\label{simplelemma}
Let $\{v_n\}_{n=1}^\infty$ be a sequence of positive continuous functions on an interval $I$, $\rho_n:=\inf_{x\in I}{v_n(x)}$.  If $v_n(x)\to v(x)$ as $n\to \infty$ locally uniformly for $x\in I$,  then  $\rho:=\inf_{x\in I}{v(x)}\geq \limsup_{n\to \infty}\rho_n$.  If $v_n(x)\to v(x)$ as $n\to \infty$ locally uniformly for $x\in I$,  then $\rho=\lim_{n\to \infty}\rho_n$. 
\end{lemma}
It is the time for us to prove our main theorem.
\begin{proof}[Proof of Theorem \ref{maintheorem}] We consider the general $g\in H(a)$ and first show that $(\tilde u(x,t;g), \tilde h(t;g))$ is an almost periodic semi-wave of \eqref{main-eq-g}.  In fact,  here we will show  that $\tilde v (\xi,\tau;g_n)\to \tilde v (\xi,\tau;g)$ locally uniformly for $(\xi,\tau)\in (-\infty,0]\times \RR$ as $n\to +\infty$ provided $g_n\to g$ in $H(a)$. 

Choosing one subsequence, we have that $v(\xi,\tau;g)=\lim_{n\to \infty}\tilde v (\xi,\tau;g_n)$ locally uniformly for $\xi\in (-\infty,0], \tau\in \RR$, and  $v=v(\xi,\tau;g)$ is one entire solution of \eqref{veq}. 
%Moreover, from above discussion,  $ v (\xi,\tau;g), \tilde v (\xi,\tau;g)\geq \epsilon$ for $\xi<-\delta, \tau\in \RR$ and $g\in H(a)$.

 We claim that the convergence here is not only locally uniform but also uniform for $\xi\leq0$ and $\tau \in [-M,M]$ where $M$ can be any positive constant. In fact, $|\tilde v(\xi,\tau;g_n)-u^*_{g_n(\cdot+\tau)}(\xi)|\to 0$ as $\xi\to -\infty$ uniformly for $\tau\in \RR$ and $n=1,2,\cdots$ by Lemma \ref{lemma2.9}. That is for any $\epsilon>0$, there is some $S>0$ such that $|\tilde v(\xi,\tau;g_n)-u^*_{g_n(\cdot+\tau)}(\xi)|<\epsilon$ for  $\xi<-S,\tau\in \RR$ and $n=1,2,\cdots$. Moreover, there is some integer $N$ such that for any $\tau,\xi\in \RR$ 
 $|u^*_{g_n(\cdot+\tau)}(\xi)-u^*_{g(\cdot+\tau)}(\xi)|<\epsilon, n>N$ since $\{u^*_g\}_{g\in H(a)}$ is a one-cover of $H(a)$. $|\tilde v(\xi,\tau;g_n)-\tilde v(\xi,\tau;g_{n'})|<4\epsilon$ for $\xi<-S,\tau\in \RR$ and $n,n'>N$. Then $|\tilde v(\xi,\tau;g_n)- v(\xi,\tau;g)|<4\epsilon$ for $\xi<-S,\tau\in \RR$ and $n>N$.  We also know $\tilde v (\xi,\tau;g_n)\to  v (\xi,\tau;g)$ uniformly for $\xi\in [-S,0], \tau\in [-M,M]$. Therefore, our claim holds. Moreover,  $\tilde v (\xi,\tau;g_n)\geq\epsilon$ for $\tau\leq\delta$ by Lemma \ref{lemma2.9} and hence $v (\xi,\tau;g)\geq\epsilon$ for $\tau\leq\delta$.  Repeating the proof of Lemma \ref{lemma2.9},   $|v (\xi,\tau;g)-u^*_{g(\cdot+\tau)}(\xi)|\to 0$ as $\xi\to -\infty$ uniformly for $\tau\in \RR$.

By Lemma \ref{lemma2.4},  $\tilde v_\xi (\xi,\tau;g_n)<0$ uniformly for $\tau\in\RR, n=1,2,\cdots$. It follows that $ v_\xi(0,\tau;g)$ and $\tilde v_\xi (0,\tau;g)$ are less than $0$ uniformly for $\tau\in \RR$.  Assume, for the sake of contradiction, that $v\not =\tilde v$.  Then for any $\tau\in \RR$, $\xi<0$,
$v(\xi,\tau;g)<\tilde v(\xi,\tau;g)$ by Proposition \ref{prop1}. Applying Proposition \ref{prop1} again, we also know that $\tilde u(t,x;g)$ is increasing in $t$. This implies that $-\tilde v_\xi+\tilde v_\tau\geq 0$. Lemma \ref{lemma10} yields that $\rho(\tau, v,\tilde v;g)$ is increasing in $\tau$.  It is also holds that $\rho(\tau, v,\tilde v;g)>0$ uniformly for $\tau\in \RR$ since both $v_\xi(0,\tau)< 0$ and $\frac{v(\xi,\tau)}{\tilde v(\xi \tau)}\to 1$ as $\xi\to -\infty$ are uniform for $\tau\in \RR$. Hence, $0<\bar \rho:=\rho(-\infty, v,\tilde v;g)<1$.

By choosing a sequence $\{\tau_n\}_{n=1}^\infty, \tau_n\to +\infty$ as $n\to \infty$, we have $g(\cdot-\tau_n)\to \bar g$ and two entire solutions $v_1(\xi,\tau;\bar g)$ and $v_2(\xi,\tau;\bar g)$ of \eqref{veq} with $g$ replaced by $\bar g$, and $v_1(\xi,\tau;\bar g)=\lim_{n\to \infty} v(\xi,\tau-\tau_n;g)=\lim_{n\to \infty}  v(\xi ,\tau;g(\cdot-\tau_n))$, $v_2(\xi,\tau;\bar g)=\lim_{n\to \infty}\tilde v(\xi, \tau-\tau_n;g)=\lim_{n\to \infty}\tilde v(\xi ,\tau;g(\cdot-\tau_n))$. Again,  the convergence is uniform for $\xi<0,\tau\in [-M,M]$, $M$ can be any positive constant and $|v_1(\xi,\tau;\bar g)-u^*_{\bar g(\cdot+\tau)}(\xi)|, |v_2(\xi,\tau;\bar g)-u^*_{\bar g(\cdot+\tau)}(\xi)|\to 0$ as $\xi\to -\infty$ uniformly for $\tau\in \RR$.  It follows Lemma \ref{simplelemma} that  $\rho(\tau, v_1,v_2;\bar g)\equiv \bar \rho$ for $\tau\in \RR$. On the other hand,  since $-\tilde v_\xi+\tilde v_\tau\geq 0$ for $\tau\in \RR$, we also have $-(v_2)_\xi+(v_2)_\tau\geq 0$ for $\tau\in \RR$.  Lemma \ref{lemma10} show that $\rho(\tau, v_1,v_2;\bar g)$ is strictly increasing in $\tau$, a contradiction.
Hence $(\tilde u(x,t;g), \tilde h(t;g))$ is one almost periodic semi-wave. We also finish the proof of the conclusions (1a) and (1b).

To prove the conclusion (1c), notice that $\tilde h(t)\to +\infty$ as $t\to +\infty$, $\tilde u(x,t;g)$  is increasing in $t$ and solves the equation $u_t=u_{xx}+u(g(x)-u)$ for $x<\tilde h(t;g)$. 
Let $\phi(x)=\lim_{t\to +\infty} \tilde u(x,t;g)$ locally uniformly for $x\in \RR$.   Then $u(t,x)\equiv \phi(x)$ also solves  $u_t=u_{xx}+u(g(x)-u)$. It also holds that $\phi(x)>0$ uniformly for $x\in \RR$. It yields that $\phi=u^*_g$.  Since $|\tilde u(y,t;g) -u^*_g (y)|\to 0$ as $y\to -\infty$ for any fixed $t$, we have $\phi(x)=\lim_{t\to +\infty} \tilde u(x,t;g)$ uniformly for $x\in(-\infty,M] $ for any constant M.   

Now we prove the conclusion (2) about the uniqueness of the almost periodic semi-wave connecting $u^*$ and $0$. Assume, for sake of contradiction, that there is another $(\tilde u_1,\tilde h_1)$  almost periodic semi-wave connecting $u^*$ and $0$ where $\tilde u_1$ is not one time translation of $\tilde u$. Then  $\tilde v_1(\xi, \tau):=\tilde u_1(\xi+(\tilde h_1)^{-1}(\tau), (\tilde h_1)^{-1}(\tau))$ is a entire solution of \eqref{veq} and there are some $\xi_0,\tau_0$ such that $\tilde v_1(\xi_0, \tau_0)<\tilde v(\xi_0, \tau_0)$. To be simple, let $\tau_0=0$. Let $\rho_0=\rho(0,\tilde v_1,\tilde v;g)$. Then $\rho_0<1$. We also have $\frac {\tilde v_1(\xi, 0)}{\tilde v(\xi, 0)}\to 1>\rho_0$  as $\xi\to -\infty$ since both of $(\tilde u,\tilde h)$ and $(\tilde u_1,\tilde h_1)$ connect $u^*$ and $0$.  Lemma \ref{lemma10} shows  $\rho(\tau,\tilde v_1,\tilde v;g)>\rho_0$ for $\tau>0$. The almost periodic property means that there is some sequence $\{\tau_n\}_{n=1}^\infty$ such that $\tilde v(\xi,\tau_n)\to \tilde v(\xi,0)$ and $\tilde v_1(\xi,\tau_n)\to \tilde v_1(\xi,0)$ locally uniformly for $\xi\leq 0$. Applying Lemma \ref{simplelemma}, we have $\frac {\tilde v_1(\xi, 0)}{\tilde v_1(\xi, 0)}\geq  \rho(\tau_n,\tilde v_1,\tilde v;g)>\rho_0$. It is a contradiction. The proof of the uniqueness is complete. 

To prove the conclusion (3), we first notice that if $(u,h)$ is one entire solution and there are some $\epsilon,\delta>0$ such that $u(x,t)>\epsilon$ for any $x\in (-\infty,h(t)-\delta], t\in \RR$,  then $h'(t)>0$ uniformly for $\tau\in \RR$. Moreover, similar as in the proof of the almost periodic property of $\tilde u$,    $| u(t,x+h(t))- u^*(x+h(t))|\to 0$ as $x+h(t)\to -\infty$ uniformly for $t\in\RR$.  Let  $v(\xi, \tau):=u(\xi+h^{-1}(\tau), h^{-1}(\tau))$. $\rho(\tau, v,\tilde v;g)$ is increasing in $\tau$.  The following proof is similar as in the above part about the almost periodic property of $\tilde v$. 
%We leave the proof of the stability in the following preciser theorem.
\end{proof}
%\begin{theorem}The semi-wave solution is globally stable in the following sense:
%\\
%For any $(u_0,h_0)$ satisfies that $h_0\in \RR$, $u_0\in C(-\infty,h_0]$ is bounded, $u_0(x)\geq 0, x\leq h_0$ and $u_0\not \equiv 0$, consider the solution $(u,h)$ of \eqref{main-eq} with the initial data $(u_0,h_0)$. 
%If $a(\cdot+h(t_n)-h_0)\to g$ as $t_n\to +\infty$, then $u(x+h(t_n)-h_0,t_n)\to \tilde u(x, t_0;g)$ with $h(t_0)=h_0$ locally uniformly for $x\leq h_0$.

%And hence $|u(x+h(t),t)- \tilde u(x+\tilde h(t';a), t';a)|\to 0 $ as $t,t'\to\infty $ locally uniformly for $x\in (-\infty,0]$ provided $h(t)=\tilde h(t';a)$.
 %\end{theorem}
 %\begin{proof}
 %We only notice that $h'(t)>0$ uniformly in $[t_0,+\infty)$ for any $t_0>0$. The rest part of the proof is similar as Theorem \ref{maintheorem} also.

 %\end{proof}

Finally, we have
\begin{corollary}
Let $(\tilde u(x,t;g), \tilde h'(t;g))$ be the sem-iwave solution of \eqref{main-eq-g}. Then the instantaneous speed $\tilde h'(t;g)$ is an almost periodic function of   $\tilde h(t;g)$ .
Mathematically, $\tilde h'(t;g)=f(\tilde h(t;g)):=\tilde u_x(\tilde h(t;g), t;g)$ with $f $ being  almost periodic. 
Moreover, $c:=\lim_{|s-r|\to \infty}\frac{|s-r|}{|\int_r^s(1/f(\tau))d\tau}$ is the average speed.
\end{corollary}

\end{document}